\documentclass[a4paper, 12pt, titlepage, fleqn]{article}
\usepackage{times}
\usepackage{amsthm}
\newtheorem{thm}{Theorem}
\newtheorem{cor}[thm]{Corollary}
\newtheorem{conj}{Conjecture}
\newtheorem{lemma}[thm]{Lemma}

\usepackage{amssymb,amsmath}

\DeclareMathOperator{\F}{\mathbb{F}}
\DeclareMathOperator{\A}{\alpha}

\DeclareMathOperator{\Tr}{Tr}

\begin{document}
\baselineskip=16.3pt 
\parskip=14pt

\begin{center}
\section*{On a Family of Twisted Trace Curves over Finite Fields, and Fibonacci Numbers}

\bigskip

{\large 
Robin Chapman \\
Department of Mathematics\\
University of Exeter\\
Exeter, EX4 4QF, UK\\
\bigskip\bigskip
Gary McGuire \\
School of Mathematics and Statistics\\
University College Dublin\\
Ireland}
\end{center}


\bigskip

\begin{center}
 In memory of Robin Chapman who sadly passed away during the course of this work. The results have been written
up by the second author and he is responsible for any errors that may be found here.
\end{center}

\bigskip

\subsection*{Abstract}

We present some results about the number of rational points on a certain family 
of curves defined over a finite field.
In a small number of cases the curves have more rational points than expected.
Fibonacci numbers make an appearance, as do cyclotomic polynomials.

\newpage
 
\section{Introduction}

Let $p$ be a prime number and let $q=p^r$ where $r$ is a positive integer.
Throughout this paper $\F_q$ will denote the finite field with $q$ elements.

Let $n\ge 4$ be a positive integer, and 
let  $d$ be a divisor of $n$ with $1<d<n$.
Let $C=C_{q,n,d}$ be the projective plane curve defined over $\F_q$ by the 
homogenization of the affine equation
\begin{equation}\label{curve1}
y+y^q+y^{q^2}+\cdots + y^{q^{n-2}}+y^{q^{n-1}}=x+x^q+x^{q^2}+\cdots + x^{q^{d-2}}+x^{q^{d-1}+q^d-1}.
\end{equation}
It is straightforward to show that  the genus of $C$
is $g(C)=(q^{n-1}-1)(q^{d-1}+q^d-2)/2$, and that $C$ has one
singular point which is the unique point at infinity on the curve.

 Let $\Tr_{q^n:q}$ denote the trace function from  $\F_{q^n}$ to $\F_q$, which is
 \[
 \Tr_{q^n:q} (x)=x+x^q+x^{q^2}+\cdots + x^{q^{n-2}}+x^{q^{n-1}}.
 \]
 The reader may have noticed that the lefthand side of \eqref{curve1} is a trace,
 and the righthand side is almost a trace.
The equation of $C=C_{q,n,d}$ can be written
\begin{equation}\label{curve2}
 \Tr_{q^n:q} (y)=x+x^q+x^{q^2}+\cdots + x^{q^{d-2}}+x^{q^{d-1}+q^d-1}.
\end{equation}
We introduce a notation for the righthand side; let
\begin{equation}\label{rhs}
R_d(x):=x+x^q+x^{q^2}+\cdots + x^{q^{d-2}}+x^{q^{d-1}+q^d-1}.
\end{equation}
The curve equation \eqref{curve1} of $C$ may  be written
\begin{equation}\label{curve3}
 \Tr_{q^n:q} (y)=R_d(x).
\end{equation}
Let $\alpha = x^{q^d-1}$.
We think of $R_d(x)$ as the trace map from $\F_{q^d}$ with the last term twisted by $\alpha$:
\[
R_d(x)=x+x^q+x^{q^2}+\cdots + x^{q^{d-2}}+\A x^{q^{d-1}}.
\]
This is why we may call $C$  a twisted trace curve.
Note that $\A=1$ if and only if $x\in \F_{q^d}^*$ if and only if $R_d(x)=\Tr_{q^d:q} (x)$.

Let $C(\F_{q^n})$ denote the set of $\F_{q^n}$-rational points of $C$.
 In this paper we will present some results about $\#C(\F_{q^n})$. There are a few surprises.
 
 For background on finite fields we refer the reader to Lidl and Niederreiter \cite{lidl}.
 
\section{First Estimate and Relation with Subfield}

\begin{lemma}\label{first}
\[
\#C(\F_{q^n}) = 1+q^{n-1}\cdot \bigl(
\text{number of $x\in \F_{q^n}$ such that  $R_d(x) \in \F_q$}\bigr).
\]
\end{lemma}

\begin{proof}
For any $x\in \F_{q^n}$ such that  $R_d(x) \in \F_q$,
there are $q^{n-1}$ values of $y$ satisfying  \eqref{curve3}, because the trace is  surjective.
Thus each such $x$ contributes $q^{n-1}$ $\F_{q^n}$-rational points.
Conversely, for any $x\in \F_{q^n}$ such that $R_d(x) \notin \F_q$,
there are no $\F_{q^n}$-rational points with that value of $x$ as its $x$-coordinate.
We add one for the point at infinity.
\end{proof}

Remark: Because $R_d(ax)=aR_d(x)$ for $a\in \F_q$, the set
of $x\in \F_{q^n}$ such that  $R_d(x) \in \F_q$ is closed under
scalar multiplication from $\F_q$.

Since $d$ divides $n$ the field $\F_{q^n}$ has a subfield $\F_{q^d}$.
This subfield gives us a number of points on the curve:

\begin{lemma}\label{subf}
$\#C(\F_{q^n}) \ge 1+q^{n-1+d}$.   
\end{lemma}

\begin{proof}
Suppose $x\in \F_{q^d}$ is nonzero.
Then $x^{q^d-1}=1$ and
$R_d(x)=\Tr_{q^d:q} (x)$.
So the righthand side of \eqref{curve3}  lies in $\F_q$ for any $x\in \F_{q^d}$.
By Lemma \ref{first} we get $\#C(\F_{q^n}) \ge 1+q^{n-1+d}$.
\end{proof}

To determine if equality holds in Lemma \ref{subf}, we 
must determine whether there are any $x\in \F_{q^n}\setminus \F_{q^d}$ with
$R_d(x) \in \F_q$.
For each such $x$ we obtain another $q^{n-1}$ rational points.
We will refer to such rational points as `bonus points'.
It is surprising that bonus points sometimes exist.
For many choices of $q, n, d$  there are no bonus points, 
and equality holds in Lemma \ref{subf}.
For example, when $n=6$ and $d=3$
we will show that there are no  bonus points for any $q$.
We will also find
some values of $q, n, d$ where there are a number of unexpected bonus points.
For example, when $n=4$ and $d=2$ we will show that
there are bonus points if and only if $q \equiv 3 \text{ mod } 5$.

In the remainder of this article we present some results of our investigation into which values of
$q, n, d$ bring bonus points, and which values do not.
We also attempt to find the exact number of bonus points. 

\section{General $d$}

As always in this paper, let  $x\in \F_{q^n}$ and let $d$ be a divisor of $n$. 
Consider again $R_d(x)=x+x^q+x^{q^2}+\cdots + x^{q^{d-2}}+x^{q^{d-1}+q^d-1}$.
Let $\alpha = x^{q^d-1}$.
We think of $R_d(x)$ as the trace map from $\F_{q^d}$ twisted by $\alpha$:
\[
R_d(x)=x+x^q+x^{q^2}+\cdots + x^{q^{d-2}}+\A x^{q^{d-1}}.
\]
Note that $\A=1$ if and only if $x\in \F_{q^d}^*$ if and only if $R_d(x)=\Tr_{q^d:q} (x)$.

\begin{thm}\label{basicrel}
Suppose $x\in \F_{q^n}$, let $\alpha = x^{q^d-1}$ and assume $\A \not= 0,1$. 
Then $R_d(x)\in \F_q$ if and only if
\begin{equation*}
(1-\A^{q+1})^{q^d-1}=\A^{q^{d-1}-1} 
(1-\A)^{q^d-1}.
\end{equation*}
\end{thm}

\begin{proof}
If $R_d(x)\in \F_q$ then $R_d(x)=R_d(x)^q$ which yields
\[
x+\alpha x^{q^{d-1}}=x^{q^{d-1}}+\alpha^{q+1} x
\]
and so
\begin{equation}\label{main1}
x(1-\A^{q+1})=x^{q^{d-1}} (1-\A).
\end{equation}
Assuming $\A \not= 1$, and $x\not= 0$, dividing both sides by $x$ and $1-\A$ gives
\begin{equation}\label{main2}
x^{q^{d-1}-1}=\frac{1-\A^{q+1}}{1-\A}.
\end{equation}
Raising \eqref{main1} to the power of $q$ gives
\begin{equation}\label{main3}
x^q(1-\A^{q^2+q})=\A x (1-\A^q)
\end{equation}
and again 
assuming $\A^{q^2+q} \not= 1$, and $x\not= 0$, dividing both sides by $x$ and 
$1-\A^{q^2+q}$ gives
\begin{equation}\label{main4}
x^{q-1}=\frac{\A(1-\A^{q})}{1-\A^{q^2+q}}.
\end{equation}
Raising this to the power of $\frac{q^d-1}{q-1}=q^{d-1}+\cdots +q+1$ 
will give $\A$ on the lefthand side.
Dividing both sides by this $\A$ (and clearing denominators) gives
\begin{equation}\label{main5}
\prod_{i=0}^{d-1} (1-\A^{q^2+q})^{q^i}=\A^{q^{d-1}+\cdots +q}
\prod_{i=0}^{d-1} (1-\A^{q})^{q^i}.
\end{equation}
Raising this to the power of $q$ gives
\begin{equation}\label{main6}
\prod_{i=0}^{d-1} (1-\A^{q^2+q})^{q^{i+1}}=\A^{q^{d}+\cdots +q^2}
\prod_{i=0}^{d-1} (1-\A^{q})^{q^{i+1}}.
\end{equation}
Dividing \eqref{main5} by \eqref{main6} gives
\begin{equation}\label{main7}
\frac{1-\A^{q^2+q}}{(1-\A^{q^2+q})^{q^d}}=\A^{q-q^d} 
\frac{1-\A^{q}}{(1-\A^{q})^{q^d}}.
\end{equation}
Applying the inverse of Frobenius to \eqref{main7} gives
\begin{equation}\label{main8}
\frac{1-\A^{q+1}}{(1-\A^{q+1})^{q^d}}=\A^{1-q^{d-1}} 
\frac{1-\A}{(1-\A)^{q^d}}.
\end{equation}
Inverting this gives the stated equation.
\end{proof}

Remarks: 

1. Note that Theorem \ref{basicrel}  holds for any $d$ and does not use $n$.

2. Because $\alpha = x^{q^d-1}$ and $x\in \F_{q^n}$, we also have
\begin{equation}\label{defa}
\A^{(q^n-1)/(q^d-1)}=1
\end{equation}
which is an equation that does use $n$.

3.  For $x\in \F_{q^n}$ the norm of $x$ is 
$N_{q^n:q} (x)=x^{(q^n-1)/(q-1)}$. We note that  $\A$ has norm 1.

4. We have two equations to work with, \eqref{main8} 
and  \eqref{defa}.  Sometimes other equations in the proof are useful instead of  \eqref{main8},
as we will see later.

5. It is easy to verify  that \eqref{main8}
is automatically satisfied if we impose the relation $\A^{q+2}=1$,
regardless of the value of $d$.
Whether the converse is true is an interesting question. 
We conjecture that it is true when $d=2$.

\begin{conj}\label{c1}
When $d=2$, equation 
\eqref{main8} implies $\A^{q+2}=1$.
\end{conj}

Even a weaker conjecture would be useful:

\begin{conj}\label{c2}
When $d=2$, equations
\eqref{main8} and  \eqref{defa}  imply $\A^{q+2}=1$.
\end{conj}

\section{The Case $d=n/2$}

\begin{thm}\label{basicrel2}
Assume $n$ is even and $d=n/2$.
If $\A \not= 0,1$ and $R_d(x)\in \F_q$ then
$\A^{D}=1$
where $D=\gcd (q^d+1,q^2-q-1)$.

Furthermore, $\#C(\F_{q^n}) = 1+q^{n-1+d}+(D-1)(q-1)q^{n-1}$.
\end{thm}

\begin{proof}
If we assume $n$ is even and $d=n/2$ 
then \eqref{defa} implies $\A^{q^d+1}=1$ or $\A^{q^d}=\A^{-1}$.
Theorem \ref{basicrel} (equation \eqref{main7}) implies
\begin{equation}\label{main12}
\frac{1-\A^{q^2+q}}{1-\A^{-q^2-q}}=\A^{q+1} 
\frac{1-\A^{q}}{1-\A^{-q}}.
\end{equation}
Using the fact that $\frac{1-x}{1-x^{-1}}=-x$ we get
\begin{equation}\label{main13}
\A^{q^2+q}=\A^{q+1} \cdot \A^q.
\end{equation}
This implies that $\A^{q^2-q-1}=1$.
We also have $\A^{q^d+1}=1$. Therefore  $\A^{\gcd (q^d+1,q^2-q-1)}=1$.

For any $\A$ satisfying $\A^D=1$ we get $q-1$ values of $x$ such
that $R_d(x)\in \F_q$, by Theorem \ref{basicrel} and \eqref{main4} in its proof.
\end{proof}

\begin{cor}\label{nob}
Assume $n$ is even and $d=n/2$.
For any values of $q$ such that\\
$\gcd (q^d+1,q^2-q-1)=1$
there are no bonus points
and $\#C(\F_{q^n}) = 1+q^{n-1+d}$.
\end{cor}

In the next section we will investigate $\gcd  (q^d+1,q^2-q-1)$.
We will see that this gcd is often equal to 1, but not always.

\subsection{The  GCD and Fibonacci Numbers}

Let $q=p^r$ where $p$ is a prime, and let $d\ge 2$ be an integer.
In this section we are going to study
$\gcd (q^d+1,q^2-q-1)$.
Although we are interested in $q$ being a prime power,
in much of the following $q$ can be any positive integer.

Denote the Fibonacci numbers by $(F_d)_{d \geq 1}$ where
$F_1=1$, $F_2=1$, and $F_d=F_{d-1}+F_{d-2}$ for $d\geq 3$.
Denote the Lucas numbers by $(L_d)_{d \geq 1}$ where
$L_1=2$, $L_2=1$, and $L_d=L_{d-1}+L_{d-2}$ for $d\geq 3$.
Recall the identities $F_{2d}=F_d L_d$ and $L_d=F_{d-1}+F_{d+1}$.

Dividing $q^2-q-1$ into $q^d+1$ it is easy to check that
\[
q^d+1=(q^2-q-1)\biggl(\sum_{i=1}^{d-1} F_i q^{d-i-1}\biggr) + F_d\ q + F_{d-1}+1
\]
and hence 
\[
\gcd (q^d+1,q^2-q-1)=\gcd(q^2-q-1, F_d\ q + F_{d-1}+1).
\]

Suppose $t$ is a prime dividing $\gcd(q^2-q-1, F_d\ q + F_{d-1}+1)$.
We can assume $t\geq 3$ because $q^2-q-1$ is always odd.
We would like to characterize such $t$.
Here is one characterization.

\begin{thm}\label{fibdiv1}
Let $d\ge 2$ be an integer.

Any  divisor of $\gcd (q^d+1,q^2-q-1)$ which is relatively prime to $F_d$
will divide \\ $F_{d+1}+F_{d-1}+1+(-1)^d$.

Conversely, let $t$ be a  divisor of $F_{d+1}+F_{d-1}+1+(-1)^d$
which is relatively prime to $F_d$.
Then $t$ divides $\gcd (q^d+1,q^2-q-1)$
for any $q \equiv - \frac{F_{d-1}+1}{F_d} \pmod{t}$.
\end{thm}

\begin{proof}
Suppose $t$ is a divisor of $\gcd(q^2-q-1, F_d\ q + F_{d-1}+1)$, and 
suppose that $\gcd (t,F_d)=1$.  Then
\[
q \equiv - \frac{F_{d-1}+1}{F_d} \pmod{t}.
\]
Since $t$ also divides $q^2-q-1$ we get $q^2\equiv q+1 \pmod{t}$ so
\[
\biggl( \frac{F_{d-1}+1}{F_d} \biggr)^2 \equiv 1-\frac{F_{d-1}+1}{F_d} \pmod{t}
\]
and clearing denominators gives
\[
(F_{d-1}+1)^2 \equiv F_d^2 -F_d(F_{d-1}+1)\pmod{t}.
\]
Using the well known Cassini identity $F_{d+1}F_{d-1}-F_d^2=(-1)^d$ we get
\[
F_{d+1}+F_{d-1}+1+(-1)^d \equiv 0 \pmod{t}.
\]
Reverse the argument to prove the converse.
\end{proof}

Remarks.

1. Since $F_{d+1}+F_{d-1} =L_d$, where $L_d$ is the $d$-th Lucas number,
$F_{d+1}+F_{d-1}+1+(-1)^d \equiv 0 \pmod{t}$  can be rewritten
\[
\begin{cases}
L_d \equiv 0 \pmod{t}\ \text{ if $d$ is odd}\\ L_d+2\equiv 0 \pmod{t} \ \text{ if $d$ is even. }
\end{cases}
\]
So prime divisors of Lucas numbers are relevant, when $d$ is odd.

2. For a given $d$ and $t$
there are infinitely many primes $q \equiv - \frac{F_{d-1}+1}{F_d} \pmod{t}$
by Dirichlet's theorem.

Example: When $d=25$, we find that
$F_{d+1}+F_{d-1}+1+(-1)^d=167761=(11)(101)(151)$.
When $q=19$ we get $\gcd (q^d+1,q^2-q-1)=11$.
When $q=23$ we get  $\gcd (q^d+1,q^2-q-1)=101$.
When $q=179$ we get $\gcd (q^d+1,q^2-q-1)=151$.

\subsection{The Case  $d$ odd}

For certain values of $d$ we can now classify when there are bonus points, and how many there are.
We give some sample results.

\begin{thm}\label{d5}
When $d=5$ and $n=10$  we have 
$\#C(\F_{q^n}) = 1+q^{n-1+d}+(G-1)(q-1)q^{n-1}$ where
\[
G=\begin{cases}
11\ \text{ if $q\equiv 8 \pmod{11}$}\\ 
1 \ \text{ otherwise. }
\end{cases}
\]
\end{thm}

\begin{proof}
Suppose $d=5$.
Then $F_5=5$ and $F_6+F_4=11$.
It is easy to check that 5 cannot divide $\gcd (q^5+1,q^2-q-1)$.
By Theorem \ref{fibdiv1}, and because 11 is a prime number, we get that 
$\gcd (q^5+1,q^2-q-1)$ is either 1 or 11. 

Let $t=11$. By Theorem \ref{fibdiv1}, take $q \equiv - \frac{F_{d-1}+1}{F_d} \pmod{t} =-(3+1)/5 \pmod{11} =8$.
So when $q \equiv 8 \pmod{11}$ we have 11 divides
$\gcd (q^5+1,q^2-q-1)$, and when $q \not\equiv 8 \pmod{11}$ we have 
$\gcd (q^5+1,q^2-q-1)=1$.
By Theorem \ref{basicrel2} the solutions to $\A^{11}=1$ apart from $\A=1$
will give bonus points.
\end{proof}

\begin{thm}\label{d7}
When $d=7$ and $n=14$  we have 
$\#C(\F_{q^n}) = 1+q^{n-1+d}+(G-1)(q-1)q^{n-1}$ where
\[
G=\begin{cases}
29\ \text{ if $q\equiv 6 \pmod{29}$}\\ 
1 \ \text{ otherwise. }
\end{cases}
\]
\end{thm}

\begin{proof}
Suppose $d=7$.
Then $F_7=13$ and $F_6+F_8=29$. 
It is easy to check that 13 cannot divide $\gcd (q^7+1,q^2-q-1)$.
By Theorem \ref{fibdiv1}, and because 29 is a prime number, we get that 
$\gcd (q^7+1,q^2-q-1)$ is either 1 or 29. Let $t=29$.
The rest is similar to the proof of Theorem \ref{d5}.
\end{proof}

\begin{thm}
When $d=9$ and $n=18$ there are at least $18(q-1)q^{n-1}$ bonus points on $C$
whenever $q \equiv 15 \pmod{19}$.
\end{thm}

\begin{proof} Let $d=9$.
Then $F_9=34$ and $F_8+F_{10}=76$. Take $t=19$.
Then take $q \equiv - \frac{F_{d-1}+1}{F_d} \pmod{t} =-(21+1)/34 \pmod{19} =15$.
So when $q \equiv 15 \pmod{19}$ we have 19 divides
$\gcd (q^9+1,q^2-q-1)$.
\end{proof}

\begin{thm}
When $d=11$ and $n=22$ there are at least $198(q-1)q^{n-1}$ bonus points on $C$
whenever $q \equiv 138 \pmod{199}$.
\end{thm}

\begin{proof} Let  $d=11$.
Then $F_{11}=89$ and $F_{10}+F_{12}=199$. Take $t=199$.
Then take $q \equiv - \frac{F_{d-1}+1}{F_d} \pmod{t} =-(55+1)/89 \pmod{199} =138$.
So when $q \equiv 138 \pmod{199}$ we have 199 divides
$\gcd (q^{11}+1,q^2-q-1)$.
\end{proof}

\vfill

\subsection{The Case $d$ even}

We remark that $F_{d+1}+F_{d-1} = F_{2d}/F_d$ and it is known
that there is always a prime dividing $F_{2d}$ that does not divide $F_d$.
So when $d$ is odd there will always be a prime 
divisor of $F_{d+1}+F_{d-1}+1+(-1)^d$
which does not divide $F_d$.
This is not true when $d$ is even, as we now prove.

\begin{lemma}
When $d$ is even every prime divisor $\not= 5$ of $F_{d+1}+F_{d-1}+1+(-1)^d$
also divides $F_d$.
\end{lemma}

\begin{proof}
First use the identity
$F_{d+1}+F_{d-1} =L_d$, where $L_d$ is the $d$-th Lucas number,
to get
$F_{d+1}+F_{d-1}+1+(-1)^d=L_d+2$.
Next use the identity $L_{4n}+2=L_{2n}^2$
to conclude that $L_d+2=L_{d/2}^2$ when $d\equiv 0 \pmod{4}$.
Finally use the identity $F_{2n}=F_n L_n$
to obtain the result when $d\equiv 0 \pmod{4}$.

When $d\equiv 2 \pmod{4}$ we need the identity
$L_{2n}+2(-1)^{n+1}=5F_{n}^2$ which gives
$L_{d}+2=5F_{d/2}^2$ and the result follows again from
$F_{2n}=F_n L_n$.
\end{proof}

\subsection{Cases where GCD is 1}

\begin{lemma}\label{qr}
All prime divisors except $5$ of $q^2-q-1$
are $\equiv \pm 1 \pmod{5}$.
\end{lemma}

\begin{proof}
By quadratic reciprocity the polynomial $x^2-x-1$ splits modulo a prime $t\not= 5$
if and only if $t \equiv \pm 1 \pmod{5}$.
\end{proof}

\begin{cor}\label{gcd1}
Given $d$, if all prime divisors of $F_d$ and $F_{d+1}+F_{d-1}+1+(-1)^d$  are 
$\equiv \pm 2 \pmod{5}$ then
$\gcd (q^d+1,q^2-q-1)=1$ for all $q$.
\end{cor}

\begin{proof}
By Theorem \ref{fibdiv1} and Lemma \ref{qr}.
\end{proof}

Using MAGMA we check that this happens for $d=3,4,8,12,16,24,32$, and more values too.
So, for example, there will never be any bonus points when $d=3$ or 4:

\begin{cor}
For $d=3$ and $n=6$ we have $\#C(\F_{q^n}) = 1+q^{n-1+d}$ for all $q$.\\
For $d=4$ and $n=8$ we have $\#C(\F_{q^n}) = 1+q^{n-1+d}$ for all $q$.
\end{cor}

\begin{proof}
When $d=3$ then $F_d=2$ and $F_{d+1}+F_{d-1}+1+(-1)^d=3+1+1-1=4$.
When $d=4$ then $F_d=3$ and $F_{d+1}+F_{d-1}+1+(-1)^d=5+2+1+1=9$.
The result follows from Corollary \ref{gcd1} and Corollary \ref{nob}.
\end{proof}

Compare this to the $d=5$ case in Theorem \ref{d5}.

\subsection{Case that $t$ divides $F_d$}

We only make a few comments.
Suppose $t$ is a prime dividing $\gcd(q^2-q-1, F_d\ q + F_{d-1}+1)$.
and suppose also that $t$ divides $F_d$.
Then $t$ divides $F_{d-1}+1$.

This can happen when $d$ is even. For example, let $q=79$, $d=30$ and $t=61$.
Then 61 divides $F_{29}+1$ and also divides $79^2-79-1$.
This means that $\gcd (79^{30}+1,79^2-79-1)$ is divisible by 61.
So there are bonus points in this case.

Suppose $d$ is even.

The identity $F_{n+1} L_n = F_{2n+1} + (-1)^n$
implies that $F_{d-1}+1=F_{d/2} L_{d/2 -1}$ when 
$d\equiv 2 \pmod{4}$.

The identity $L_{n+1} F_n = F_{2n+1} - (-1)^n$
implies that $F_{d-1}+1=L_{d/2} F_{d/2 -1}$ when 
$d\equiv 0 \pmod{4}$.

However this does not seem to happen when $d$ is odd. Suppose $d$ is odd. 
We conjecture that the GCD of $F_d$ and $F_{d-1}+1$ is either 1 or 2.
This is surely known already.

\newpage

\section{The Case $d=2$}

Assume that $n$ is even and $d=2$.
We have the relation $\A^{(q^n-1)/(q^2-1)}=1$.

\begin{thm}\label{d2}
Assume that $n$ is even and $d=2$.
Furthermore we assume that $\A^{q+2}=1$ (i.e. assume Conjecture 1).
Then $C$ will have
at least $(G-1)\cdot (q-1)\cdot q^{n-1}$ bonus points, where
$G=\gcd (q^n-1,H)$
and $H=\gcd (q+2,(2^n-1)/3)$.
\end{thm}

\begin{proof}
Taking the $q$-th power of the relation $\A^{q}=\A^{-2}$ 
gives $\A^{q^2}=\A^{-2q}=(\A^q)^{-2}=\A^4$.
Raising this to the power of $q^2$ gives 
$\A^{q^4}=(\A^{q^2})^{4}=\A^{16}$.
By induction we get $\A^{q^r}=\A^{2^r}$ for any even $r$.
Then
\[
1=\A^{(q^n-1)/(q^2-1)}=\A^{1+q^2+q^4+\cdots +q^{n-2}}=
\A^{1+2^2+2^4+\cdots +2^{n-2}}=\A^{(2^n-1)/(2^2-1)}=\A^{(2^n-1)/3}.
\]
For every $\A$ satisfying  $\A^{q+2}=1$ and $\A^{(2^n-1)/3}=1$ we have $\A^H=1$.
For every $\A\in \F_{q^n}$ satisfying  $\A^H=1$ we have $\A^G=1$, and conversely.

By  \eqref{main4}  for each $\A\not= 1$ we obtain $q-1$ values of $x$ with $R_d(x)\in \F_q$.
\end{proof}

\subsection{The Case  $d=2$ and $n=4$}

\begin{cor}
For $d=2$ and $n=4$ we have $\#C(\F_{q^n}) = 1+q^{n-1+d}+4(q-1)q^{n-1}$ for all $q\equiv 3 \pmod{5}$,
and $\#C(\F_{q^n}) = 1+q^{n-1+d}$ otherwise.
\end{cor}

\begin{proof}
We apply Theorem \ref{d2}.
First we prove that Conjecture \ref{c1} is true when $n=4$ and $d=2$.
We may use the results from the $d=n/2$ case to do this.
From the proof of Theorem \ref{basicrel2} we get $\A^{q^2-q-1}=1$
and $\A^{q^2+1}=1$.
These imply $\A^{q+2}=1$.

We now apply Theorem \ref{d2}. When $n=4$ we get $(2^n-1)/3=5$ and
\[
H=\gcd (q+2,5)=\begin{cases}
5\ \text{ if $q\equiv 3 \pmod{5}$}\\ 1 \ \text{ otherwise. }
\end{cases}
\]
Then 
$G=\gcd (q^4-1,H)=H$.
\end{proof}

\subsection{The Case  $d=2$ and $n=6$}

Unfortunately we cannot use results from the $d=n/2$ case here to prove Conjecture 1,
and we do not have a proof of Conjecture 1.
If we assume the conjecture, we easily get the following result.

\begin{cor}\label{n6d2}
Assume Conjecture \ref{c1} is true.
For $d=2$ and $n=6$ we have \\
$\#C(\F_{q^n}) = 1+q^{n-1+d}+(G-1)(q-1)q^{n-1}$ where
\[
G=\begin{cases}
21\ \text{ if $q\equiv 19 \pmod{21}$}\\ 
7\ \text{ if $q\equiv 5 \pmod{7}$ and $q\not\equiv 1 \pmod{3}$}\\ 
3\ \text{ if $q\equiv 1 \pmod{3}$ and $q\not\equiv 5 \pmod{7}$}\\
1 \ \text{ otherwise. }
\end{cases}
\]
\end{cor}

\begin{proof}
We apply Theorem \ref{d2}.
When $n=6$ we get $(2^n-1)/3=21$ and
\[
H=\gcd (q+2,21)=\begin{cases}
21\ \text{ if $q\equiv 19 \pmod{21}$}\\ 
7\ \text{ if $q\equiv 5 \pmod{7}$ and $q\not\equiv 1 \pmod{3}$}\\ 
3\ \text{ if $q\equiv 1 \pmod{3}$ and $q\not\equiv 5 \pmod{7}$}\\
1 \ \text{ otherwise. }
\end{cases}
\]
Then  $G=H$ because the 3-rd and 7-th roots of unity are in $\F_{q^6}$.
\end{proof}

\subsection{The Case  $d=2$ and $n=8$}

If we assume Conjecture 1, we obtain a similar result to Corollary \ref{n6d2} using  Theorem \ref{d2}.
When $n=8$ we get $(2^n-1)/3=85$ so the 85-th roots of unity will appear.

\newpage

\section{A Different Approach, $d=2$ and $n=6$}

We will present a proof of Corollary  \ref{n6d2} without the assumption that Conjecture \ref{c1} is true.
The proof is by a different approach, requires a computer algebra package, and is somewhat ad hoc.
The proof of Theorem \ref{factors3} is, in some sense, a proof of Conjecture  \ref{c1} in this case.

Assume for this section that  $d=2$ and $n=6$.
Suppose that $x\in \F_{q^6}$ and that $R_d(x)\in \F_{q}$.
Let $\alpha = x^{q^2-1}$.

We begin with equation \eqref{main2} which says
\[
x^{q-1}=\frac{\alpha^{q+1}-1}{\alpha -1}.
\]
Raising both sides to the power of $q+1$ gives
\begin{equation}\label{n6d21}
\alpha (\alpha -1)(\alpha^q -1)=(\alpha^{q+1} -1)(\alpha^{q^2+q} -1).
\end{equation}

\subsection{Motivation}

We use a well-known multivariate method.
If we let  $Y_i = \alpha^{q^i}$ for $i=0,1,2,3,4,5$,
then \eqref{n6d21} becomes
\begin{equation}\label{n6d22}
Y_0(Y_0-1)(Y_1-1)=(Y_0Y_1-1)(Y_1Y_2-1)
\end{equation}
Solving this for $Y_2$ gives 
\begin{equation}\label{n6d23}
Y_2=\frac{Y_0^2Y_1-Y_0^2+Y_0-1}{Y_1(Y_0Y_1-1)}.
\end{equation}
Applying Frobenius gives
\begin{equation}\label{n6d24}
Y_{i+2}=\frac{Y_i^2Y_{i+1}-Y_i^2+Y_i-1}{Y_{i+1}(Y_iY_{i+1}-1)} \ \ \ \text{ for } i=1,2,3.
\end{equation}
Next we use equation \eqref{defa} which in this case is
\begin{equation}\label{defa2}
\A^{(q^6-1)/(q^2-1)}=1
\end{equation}
or $\alpha^{q^4+q^2+1}=1$. With the new variables this is $Y_0Y_2Y_4=1$.
Applying Frobenius gives $Y_1Y_3Y_5=1$.

\subsection{The proof}

Motivated by the previous section
we introduce new variables $y_0$ and $y_1$, and we define $y_2, y_3, y_4, y_5$ by
 \eqref{n6d23} and \eqref{n6d24} (with all $Y$ replaced by $y$), e.g.
 \[
 y_2:=\frac{y_0^2y_1-y_0^2+y_0-1}{y_1(y_0y_1-1)}.
 \]
All coefficients are integers.

Substituting for the variables using \eqref{n6d23} and \eqref{n6d24} into $y_0y_2y_4-1$ gives us a 
rational function in two variables $y_0$ and $y_1$ with integer coefficients.  
Denote the numerator of this rational function by
$F_1(y_0,y_1)$. We computed this with MAGMA \cite{magma}:
\[
F_1(y_0,y_1)=-y_0^{11}y_1^3 + 3y_0^{11}y_1^2 - 3y_0^{11}y_1 + y_0^{11} + y_0^{10}y_1^5 - 2y_0^{10}y_1^4 + 3y_0^{10}y_1^3 -    9y_0^{10}y_1^2 + 12y_0^{10}y_1 - \]
    \[ 5y_0^{10} - y_0^9y_1^5 + 
   3y_0^9y_1^4 - 2y_0^9y_1^3 +
    11y_0^9y_1^2 - 24y_0^9y_1 + 13y_0^9 - 4y_0^8y_1^4 + 3y_0^8y_1^3 - 7y_0^8y_1^2 +
    30y_0^8y_1 - 22y_0^8 + \]
    \[ 3y_0^7y_1^4 - 4y_0^7y_1^3 - y_0^7y_1^2 - 23y_0^7y_1 + 26y_0^7 -
    y_0^6y_1^9 + 4y_0^6y_1^3 + 4y_0^6y_1^2 + 9y_0^6y_1 - 
    22y_0^6 + y_0^5y_1^{10} + 4y_0^5y_1^8 -\]
\[    3y_0^5y_1^3 - 2y_0^5y_1^2 + 2y_0^5y_1 + 13y_0^5 - 5y_0^4y_1^9 + y_0^4y_1^8 - 6y_0^4y_1^7 -
    5y_0^4y_1 - 5y_0^4 + 10y_0^3y_1^8 - 4y_0^3y_1^7 + 4y_0^3y_1^6 + \]
    \[ y_0^3y_1^2 + 3y_0^3y_1 +
    y_0^3 - 10y_0^2y_1^7 + 6y_0^2y_1^6 - y_0^2y_1^5 - y_0^2y_1 + 5y_0y_1^6 - 4y_0y_1^5 - y_1^5 +
    y_1^4
    \]
    and the denominator is
    \[
    y_0^6y_1^9 - y_0^5y_1^{10} - 4y_0^5y_1^8 + 5y_0^4y_1^9 - y_0^4y_1^8 + 6y_0^4y_1^7 -
    10y_0^3y_1^8 + 4y_0^3y_1^7 - 4y_0^3y_1^6 + 10y_0^2y_1^7 - \] \[ 6y_0^2y_1^6 + y_0^2y_1^5 -
    5y_0y_1^6 + 4y_0y_1^5 + y_1^5 - y_1^4.
\]

Similarly, $y_1y_3y_5-1$ is a rational function in two variables $y_0$ and $y_1$ with integer coefficients.
Denote the numerator of this rational function by
$F_2(y_0,y_1)$.

The resultant of  $F_1(y_0,y_1)$ and $F_2(y_0,y_1)$ with respect to $y_0$ is a polynomial in $y_1$, 
which has integer coefficients and has degree 240.
Call this $G_1(y_1)$.
Using MAGMA or MAPLE we can factor $G_1(y_1)$  into irreducible factors,
and we note that the coefficients are all divisible by 2, and no other prime. 
We get
\[
G_1(y_1)=4y_1^{76}\ p_1(y_1)^{110}\  p_2(y_1)^6\ p_3(y_1)^8\ p_4(y_1)^2\ p_5(y_1)\ p_6(y_1)
\]
where
$p_1(y_1)=y_1-1$\\
$p_2(y_1)=y_1^2+y_1+1$\\
$p_3(y_1)=y_1^2-y_1+1$\\
$p_4(y_1)=y_1^4-y_1^3+2y_1^2-2y_1+1$\\
$p_5(y_1)=y_1^6+y_1^5+y_1^4+y_1^3+y_1^2+y_1+1$\\
$p_6(y_1)=y_1^{12}-y_1^{11}+y_1^9-y_1^8+y_1^6-y_1^4+y_1^3-y_1+1$.

Notice that $p_2(y_1)$ is the 3rd cyclotomic polynomial (commonly denoted $\Phi_3(y_1)$),
$p_3(y_1)$ is the 6-th cyclotomic polynomial $\Phi_6(y_1)$,
$p_5(y_1)$ is the 7-th cyclotomic polynomial $\Phi_7(y_1)$, and $p_6(y_1)$ is the 21-st cyclotomic polynomial $\Phi_{21}(y_1)$.
This will be relevant later.

The resultant of  $F_1(y_0,y_1)$ and $F_2(y_0,y_1)$ with respect to $y_1$ is a polynomial in $y_0$, which has integer coefficients and has degree 344.  Call this $G_2(y_0)$.
Using a computer algebra package we  factor $G_2(y_0)$  into irreducible factors,
and note that the coefficients are all divisible by 2, and no other prime. 
We get
\[
G_2(y_0)=4y_0^{149}\ p_1(y_0)^{101}\  p_2(y_0)^6\ p_3(y_0)^{20}\ p_4(y_0)^4\ p_5(y_0)\ p_6(y_0)\ p_7(y_0)
\]
where
$p_7(y_0)=y_0^8-y_0^7+4y_0^6 -7y_0^5+6y_0^4-5y_0^3+4y_0^2-2y_0+1$.

Since both resultants are 0 in characteristic 2, and are nonzero in odd characteristic, we suppose from now on that $q$ is odd.

\begin{lemma}\label{factors1}
Assume  that  $d=2$ and $n=6$, and that $q$ is odd.
Suppose that $x\in \F_{q^6}$ and that $R_d(x)\in \F_{q}$.
Let $\alpha = x^{q^2-1}$.
Then $\alpha$ is a common root of the polynomials $G_1$ and $G_2$.
\end{lemma}

\begin{proof}
Since the $F_i$ and $G_i$ have integer coefficients, we consider their reductions modulo $p$
and we use the same notation.
By construction we have $F_1(\A, \A^q)=0$ and $F_2(\A, \A^q)=0$.
Also by construction, $\alpha$ is a root of $G_1$ and  $\alpha^q$ is a root of $G_2$.
By Galois action, $\alpha$ is also a root of $G_2$.
\end{proof}

\begin{lemma}\label{factors2}
Assume  that  $d=2$ and $n=6$, and that $q$ is odd.
Suppose that $x\in \F_{q^6}$ and that $R_d(x)\in \F_{q}$.
Let $\alpha = x^{q^2-1}$.
Then $\alpha$ is a $21$-st root of unity.
\end{lemma}

\begin{proof}
By the discussion above and Lemma \ref{factors1},
any common root of $G_1$ and $G_2$ that is not a 21-st root of unity must be
a root of  $p_3$ or  $p_4$.
So we have to show that if $\A$ is a root of $p_3$ or $p_4$
lying in $\F_{q^6}$ then $\A$ cannot satisfy \eqref{n6d21} (or possibly $\A$ is a 21-st root of unity).

Assume $\A\not= 1$.

First we shall deal with $p_3$. 
For $p_3$, a quadratic polynomial,
the coefficients are $\pm 1$ and obviously this polynomial has its roots in either 
$\F_q$ or $\F_{q^2}$.
Suppose that $\alpha$ is a primitive 6-th root of unity and that $\A$ satisfies \eqref{n6d21}.
We have $\alpha^q=\alpha^t$ where $t=1$ or 5.

Case 1: $t=1$.
Then $\alpha^q=\alpha$ and $\A\in \F_q$. So $\A^{q+1}=\A^2$ and \eqref{n6d21} becomes
\[
\A (\A-1)^2=(\A^2-1)^2
\]
or $\A^2+\A+1=0$. This says that $\A$ is a cube root of unity, a contradiction.

Case 2: $t=5$.
Then $\alpha^q=\alpha^5=\A^{-1}$ and 
 $\A^{q+1}=\A^{-1} \A=1$. Equation \eqref{n6d21} implies $\A=0$, a contradiction.

Secondly we shall deal with $p_4$.
Consider the reduction of $p_4$ modulo $p$.
If $p_4$ is irreducible over $\F_q$ then its roots are in $\F_{q^4}$ 
(and no subfield) and therefore are not in $\F_{q^6}$.

Suppose now that $p_4$ is reducible over $\F_q$.
We will show that  $p_4$ must factor into irreducible factors
over $\F_{q}$ into linear and quadratic factors.
The Galois group of $p_4$ over $\mathbb{Q}$ is the dihedral group of order 8.
If $p_4$ factored mod $p$ as a linear factor times an irreducible cubic, then there
would be an element of order 3 in the Galois group, by Dedekind's theorem.
Since the dihedral group of order 8 does not have an element of order 3, this is impossible.
Therefore, all roots of $p_4$ lie in $\F_{q^2}$.

Suppose $\A$ is a root of $p_4$ in $\F_{q^2}$.
Then $\A^{q^2-1}=1$ which implies $\A^{q+1}\in \F_q$.
If follows that $(\alpha^{q+1} -1)(\alpha^{q^2+q} -1)$, the righthand side of \eqref{n6d21},
lies in $\F_q$.
Also $(\A-1)(\A^q-1)$ is in $\F_q$ because Frobenius interchanges $\A$ and $\A^q$.
By  \eqref{n6d21}, $\A\in \F_q$.

Suppose $\A$ is a root of $p_4$ in $\F_q$.
As in Case 1 earlier in the proof,  \eqref{n6d21}  implies  that $\A^2+\A+1=0$. 
So $\A$ is a cube root of unity, and also a 21-st root of unity.
The proof is complete.
\end{proof}

Remarks.

1. We continue the last paragraph of the proof to get more information.
Combining  $\A^2+\A+1=0$ with the equation $p_4(\A)=\A^4-\A^3+2\A^2-2\A+1=0$ gives
$2\A^2=\A$. Since $\A\not=0$ this implies $2\A=1$.
Then $\A^2+\A+1=7/4$. This is not 0 unless we are in characteristic 7.
So assume the characteristic is 7.
Then $\A=1/2=4$.
This is the only case where a root of $p_4$ is also a 21-st root of unity.

2. We have proved some constraints on $\A$, however we still have not proved that
Conjecture 1 holds ($\A^{q+2}=1$).

Here is our proof of Corollary  \ref{n6d2} without assuming Conjecture 1.
Conjecture 1 is proved during the course of this proof.

\begin{thm}\label{factors3} 
Assume  that  $d=2$ and $n=6$, and that $q$ is odd.  Then \\
$\#C(\F_{q^n}) = 1+q^{n-1+d}+(G-1)(q-1)q^{n-1}$ where
\[
G=\begin{cases}
21\ \text{ if $q\equiv 19 \pmod{21}$}\\ 
7\ \text{ if $q\equiv 5 \pmod{7}$ and $q\not\equiv 1 \pmod{3}$}\\ 
3\ \text{ if $q\equiv 1 \pmod{3}$ and $q\not\equiv 5 \pmod{7}$}\\
1 \ \text{ otherwise. }
\end{cases}
\]
\end{thm}

\begin{proof}
Suppose that $x\in \F_{q^6}$ and that $R_d(x)\in \F_{q}$.
Let $\alpha = x^{q^2-1}$.  Assume $\A\not= 1$.

By the previous lemmas and discussion we simply need to list the possible $\A$ 
and corresponding $x$ satisfying all the following constraints:

$x^{q-1}=\frac{\alpha^{q+1}-1}{\alpha -1}$

$\alpha (\alpha -1)(\alpha^q -1)=(\alpha^{q+1} -1)(\alpha^{q^2+q} -1)$ (Equation \eqref{n6d21})

$\A^{21}=1$.

We will consider three cases:

1. $\A$ is a primitive cube root of unity

2. $\A$ is a primitive 7-th root of unity

3. $\A$ is a primitive 21-st root of unity

\underline{Case 1.} We will show that $q$ must be 1 mod 3. 
If $q$ is 0 mod 3 then $\A^q=1$ and \eqref{n6d21} gives a contradiction.
If $q$ is 2 mod 3 then $\A^{q+1}=1$ and \eqref{n6d21} gives a contradiction.
If $q$ is 1 mod 3 then  \eqref{n6d21} holds, and 
$x^{q-1}=\frac{\alpha^{q+1}-1}{\alpha -1}=\frac{\alpha^{2}-1}{\alpha -1}=\A +1=-\A^2=-\A^{-1}$ has $q-1$ solutions for $x$
by Hilbert Theorem 90 (because $-\A^{-1}$ has norm 1).

\underline{Case 2.} We know $\A^q=\A^t$ where $t\in \{1,2,3,4,5,6\}$. We consider each of these values of $t$ separately.
When $t=5$, it is easy to check that equation \eqref{n6d21} automatically holds, and we will see that all other values of $t$ will give a contradiction. 

When $t=5$ then $\A^q=\A^5$ so $q\equiv$ 5 mod 7. In this case we get
\[
x^{q-1}=\frac{\alpha^{q+1}-1}{\alpha -1}=\frac{\alpha^{6}-1}{\alpha -1}=\frac{\alpha^{-1}-1}{\alpha -1}=-\A^{-1}.
\]
Each $\A$ will give $q-1$ solutions for $x$ by Hilbert Theorem 90 (because $-\A^{-1}$ has norm 1).

Case $t=1$. Then $\A^q=\A$ and \eqref{n6d21} becomes $\A^2+\A+1=0$, so $\A$ is a cube root of unity,
which is impossible by our hypotheses.

Case $t=2$. Then \eqref{n6d21} becomes
$\alpha (\alpha -1)(\alpha^2 -1)=(\alpha^{3} -1)(\alpha^{6} -1)$ which gives\\
$\A^6+\A^4-2\A^2+\A-1=0$.
Factoring out $\A-1$ gives $\A^5+\A^4+2\A^3+2\A^2+1=0$.
Subtracting this from  $\A^6+\A^5+\A^4 +\A^3+\A^2+\A+1=0$ and dividing by $\A$ gives\\
$\A^5-\A^2-\A+1=0$.
Factoring out $\A-1$ gives $\A^4+\A^3+\A^2-1=0$.
Subtracting $\A$ times this from $\A^5+\A^4+2\A^3+2\A^2+1=0$ gives
$\A^3+2\A^2+\A+1=0$.
Adding this to $\A^4+\A^3+\A^2-1=0$ and dividing by $\A$ gives $\A^3+2\A^2+3\A+1=0$.
Subtracting these two cubics gives $2\alpha=0$. Since $q$ is odd, this is impossible.

Cases $t=3,4,6$ are similar and we omit the details.

\underline{Case 3.} We know $\A^q=\A^t$ where $1\le t \le 20$ and $\gcd (t,21)=1$.
When $t=19$, it is easy to check that equation \eqref{n6d21} automatically holds, and we must show that all other values of $t$ will give a contradiction. 

When $t=19$ then $\A^q=\A^{19}$ so $q\equiv$ 19 mod 21. In this case we get
$x^{q-1}=-\A^{-1}$ and
each $\A$ will give $q-1$ solutions for $x$ by Hilbert Theorem 90.

We omit the details of the other values of $t$.
\end{proof}

\section{Final Remarks}

Returning to Remark 5 at the end of Section 3,
Conjecture 1 states that 
\eqref{main8} implies  $\A^{q+2}=1$, when $d=2$ and $n$ is even. We are unable to prove this conjecture,
but somehow it is important to this problem. 
A full investigation of the implications of Conjecture 1, as well as its proof, is a topic for 
future work. Also, generalizations of Conjecture 1 for $d>2$ are probably available.

The motivation for Conjecture 1 is that if $\A^{q+2}=1$, then the equation in the statement of Theorem 3 is automatically satisfied.
Furthermore, 
equation \eqref{main2} which says
\[
x^{q-1}=\frac{\alpha^{q+1}-1}{\alpha -1}.
\]
becomes $x^{q-1}=-\A^{-1}$, which will be familiar from the proof of Theorem \ref{factors3}.
Since $-\A^{-1}$ has norm 1 this equation always has $q-1$ solutions for $x$ by Hilbert Theorem 90.

The proof of Theorem \ref{factors3} can be seen as a proof of Conjecture 1 in the
case $n=6$ and $d=2$.

\end{document}